\documentclass{amsart}

\usepackage{amsfonts,amsmath,amssymb}
\usepackage{latexsym,amscd,graphicx}

\newcommand{\bg}{{\overline{g}}}

\newcommand{\bH}{{\overline{H}}}

\newcommand{\vp}{\varphi}
\newcommand{\ld}{\ldots}

\newcommand{\sgn}{\mathrm{sgn}}

\DeclareMathOperator*{\ot}{\otimes}
\DeclareMathOperator*{\op}{\oplus}

\newcommand{\beq}{\begin{equation}}
\newcommand{\eeq}{\end{equation}}
\newcommand{\beas}{\begin{eqnarray*}}
\newcommand{\eeas}{\end{eqnarray*}}
\newcommand{\id}{\mathrm{id}}

\newcommand{\cD}{\mathcal{D}}

\newcommand{\NN}{\mathbb{N}}
\newcommand{\ZZ}{\mathbb{Z}}

\newcommand{\CC}{\mathbb{C}}

\newcommand{\FF}{\mathbb{F}}

\newcommand{\chr}[1]{\mathrm{char}\,#1}

\newcommand{\hf}{H_{\mathrm{fin}}}

\DeclareMathOperator{\End}{\mathrm{End}}

\newcommand{\ad}{\mathrm{ad}}

\newcommand{\ch}{\mathrm{char}}

\newcommand{\Span}{\mathrm{Span}}

\numberwithin{equation}{section}

\newtheorem{theorem}[equation]{Theorem}

\newtheorem{lemma}[equation]{Lemma}
\newtheorem{conjecture}[equation]{Conjecture}
\newtheorem{corollary}[equation]{Corollary}
\newtheorem{proposition}[equation]{Proposition}
\newtheorem{definition}[equation]{Definition}

\theoremstyle{definition}
\newtheorem{example}[equation]{Example}

\newcommand{\gr}{\mathrm{gr}}

\begin{document}

\title{Delta sets and polynomial identities in pointed Hopf algebras}

\author[Bahturin]{Yuri Bahturin}
\address{Department of Mathematics and Statistics, Memorial
University of Newfoundland, St. John's, NL, A1C5S7, Canada}
\email{bahturin@mun.ca}

\author[Witherspoon]{Sarah Witherspoon}
\address{Department of Mathematics, Texas A\&M University, College
Station, Texas 77843, USA}
\email{sjw@math.tamu.edu}

\thanks{The first author acknowledges a partial support by NSERC Discovery Grant 2019-05695. The second author acknowledges partial support by NSF grants DMS-1665286
and DMS-2001163.}

\begin{abstract} We survey a vast array of known results and techniques in the area of polynomial identities in pointed Hopf algebras. Some new results are proven in the setting of  Hopf algebras that appeared in papers of D.~Radford and N.~Andruskiewitsch - H.-J.~Schneider.

\end{abstract}

\maketitle

\section{Introduction}\label{s1}

In this paper our main concern is the determination of conditions under which a pointed Hopf algebra $H$ over a field $\FF$ is a PI-algebra, that is $H$ satisfies a nontrivial polynomial identity $f(x_1,\ld,x_n)=0 $. Here $f(x_1,\ld,x_n)$ is a nonzero element of the free associative algebra (i.e.~algebra of noncommutative polynomials) in the variables $x_1,\ld,x_n$. For an excellent modern source of information about PI-algebras we refer the reader to Giambruno - Zaicev's monograph \cite{GZ}. As a crash course, here we recall a few facts about PI-algebras that are relevant to our main problem, particularly some of the many results known about group algebras, universal enveloping algebras, and smash products. 

Our main new results are:
(1)  A generalization in Theorem~\ref{tLSABW}, to color Lie superalgebras, of Kochetov's classification of Lie superalgebras and their smash products satisfying polynomial identities.
(2) A classification in Theorem~\ref{A-S-PI} of a large class of pointed Hopf algebras, arising in work of Andruskiewitsch - Schneider, satisfying polynomial identities. 
Along the way we also discuss delta sets in pointed Hopf algebras. 
They play a crucial role in PI theory and also prove useful in other situations.

We begin in Sections~\ref{ssPI} and~\ref{ssCCHA} by recalling some definitions and summarizing some of the general theory we will need as well as some known results about cocommutative Hopf algebras, particularly group algebras, universal enveloping algebras of Lie algebras, and their smash products. In Section~\ref{ssNCCHA} we state our generalization of Kochetov's work to color Lie superalgebras: In Theorem~\ref{tLSABW} we give necessary and sufficient conditions for $U(L)\# \FF G$ to be a PI-algebra, where $U(L)$ is the universal enveloping algebra of a color Lie superalgebra $L$, $G$ is an acting group, and $\FF$ has characteristic~0. In the case of positive characteristic, these conditions are not known. The proof of Theorem~\ref{tLSABW} is deferred to Section~\ref{sPT}; we show there that Kochetov's proof generalizes with no obstacles. 

To prepare for the proof of Theorem~\ref{tLSABW} as well as to set the stage for understanding a further large class of pointed Hopf algebras introduced by Andruskiewitsch and Schneider, in Section~\ref{sDS} we collect some more specialized known techniques for group algebras, Lie superalgebras, and smash products. We recall delta sets and known results for groups, Lie algebras, and Hopf algebras, and adapt them to color Lie superalgebras. 

Radford's Hopf algebras $F_{(q)}$ feature in Section~\ref{sFQ}; these are Hopf algebras generated by one grouplike and one skew-primitive element depending on a scalar $q$. They play a key role as subalgebras in more general pointed Hopf algebras, and our work in Section~\ref{sFQ} is accordingly called upon later. We show that $F_{(q)}$ is a PI-algebra if and only if $q$ is a root of unity, and we give further results on the degree of the corresponding polynomial identity as well as some results on delta sets for $F_{(q)}$.

In Section~\ref{pointed} we define Hopf algebras $U(\cD,\lambda)$ depending on some data $\cD,\lambda$; these are the pointed Hopf algebras arising in Andruskiewitsch and Schneider's classification of finite dimensional pointed Hopf algebras, but with more general groups of grouplikes allowed. In Section~\ref{ssIR} we give in Theorem~\ref{A-S-PI} necessary and sufficient conditions for $U(\cD,\lambda)$ to be a PI-algebra and study delta sets.

Throughout, $\FF$ will be a field, of arbitrary characteristic unless stated otherwise.

\section{PI-algebras}\label{ssPI}
Note that every finite-dimensional algebra $R$ over a field $\FF$ is PI: if $\dim_{\FF}\! R=n-1$, then $R$ satisfies the so called \textit{standard identity}:
\begin{equation}\label{eSt}
s_n(x_1,\ld,x_n)=\sum_{\sigma\in\mathrm{Sym}(n)}\sgn(\sigma)x_{\sigma(1)}\cdots x_{\sigma(n)}=0.
\end{equation}
At the other extreme, commutative algebras of any dimension satisfy $s_2(x,y)=xy-yx=[x,y]=0$. Generally, for any PI-algebra $R$ there exist numbers $m$ and $n$ such that $s_n(x_1,\ld,x_n)^m=0$ is an identity in $R$. 

Subalgebras, factor-algebras and extensions (including direct products) of PI-algebras are PI. An important theorem of A. Regev \cite{AR} says the following.
\begin{theorem}\label{tAR}
If $A$ and $B$ are PI-algebras over an arbitrary field $\FF$, then $A\ot_\FF B$ is a PI-algebra.
\end{theorem}
In particular, if $n$ is a natural number and $S$ is a PI-algebra then $M_n(S)=M_n(\FF)\ot_{\FF} S$ is a PI-algebra. One of the frequent applications of this is to extensions of the field of coefficients to other fields or even to commutative algebras $S$ over $\FF$. In all these cases the algebra obtained by an extension of coefficients from a PI-algebra remains PI.

Another corollary of the techniques of Regev's Theorem, which is useful when dealing with the quantum commutator structure of an algebra, is the following (see \cite[Proposition 4.1.11]{BMPZ}). In what follows, $[a,b]_q=ab-qba$.
\begin{proposition}\label{pq_id}
Suppose that an associative algebra $A$ over a field $\FF$ satisfies a nontrivial polynomial identity of degree $d$. 
Then for any $q\in\FF$, $A$ satisfies a nontrivial identity of the form
\begin{equation}\label{eq_id}
f^q(x_1,\ld,x_n,y_1,\ld,y_n)=\sum_{\sigma\in\mathrm{Sym}(n)}\lambda_\sigma[x_1,y_{\sigma(1)}]_q\cdots[x_n,y_{\sigma(n)}]_q=0,
\end{equation}
where $n=3d^4$ and  
$\lambda_\sigma\in\FF$.
\end{proposition}

The following theorem covers PI-algebras that are also a Hopf algebras.
Recall the adjoint action of $H$ on itself:  $(\ad x)(y) = \sum x_1 y S(x_2)$ where $S$ is the antipode and comultiplication takes $x$ to $\sum x_1\ot x_2$.
\begin{theorem}\label{tH_inner_id}
Suppose that a Hopf algebra $H$ over a field $\FF$ satisfies a nontrivial polynomial identity of degree $d$. 
Then $H$ satisfies a nontrivial identity of the form
\begin{equation}\label{eH_inner_id}
f(x_1,\ld,x_n,y_1,\ld,y_n)=\sum_{\sigma\in\mathrm{Sym}(n)}\lambda_\sigma(\ad\,x_1)(y_{\sigma(1)})\cdots(\ad\,x_n)(y_{\sigma(n)})=0,
\end{equation}
where $n=3d^4$ and  
$\lambda_\sigma\in\FF$.
\end{theorem}

In \cite{MKP} the proof of this theorem (attributed by M. Kochetov to Y. Bahturin) is given for smash products of the form $A\# H$, where $A$ is an $H$-module algebra. In that case $x_1,\ld,x_n\in H$ and $y_1,\ld,y_n\in A$. However, it works in our more general situation, without any changes. The estimate $n=3d^4$ is borrowed from \cite[Lemma 4.1]{BP}.

Being PI is ``almost'' equivalent to having an upper bound on the dimension of simple modules. On the one hand, a direct calculation shows that a matrix algebra $M_n(K)$ of order $n$ over an $\FF$-algebra $K$ does not satisfy a nontrivial identity of degree less than $2n$. A much sharper result is a theorem by Amitsur - Levitzky stating that $s_{2n}(x_1,\ld,x_{2n})=0$ holds in $M_n(\FF)$ and there are no identities of degree less than $2n$. Now using the Density Theorem, we can easily see that, for any $n$, if an algebra $A$ has a simple module of dimension $\ge n$ over its centralizer $D=\End_AV$ then $A$ does not satisfy any identity of degree less than $2n$.  

On the other hand, if $A$ is semiprimitive and all simple (left, right) modules are of dimension at most $n$ over $\FF$ then $A$ is a PI-algebra (satisfying  $s_{2n}(x_1,\ld,x_{2n})=0$). Indeed such an algebra is a subCartesian product of primitive algebras $\End_DV$. If $\dim_\FF D=\ell$ and $\dim_D V=m$, with $\ell m=n$, then $A$ is a subCartesian product of simple algebras of dimension $\ell m^2=mn$. If $\overline{\FF}$ is a splitting field for $D$ then $\overline{\FF}\ot \End_DV\cong M_{\sqrt{mn}}$, so $s_{2n}(x_1,\ld,x_{2n})=0$ must hold.

We complete this introduction by citing two important theorems about the structure of PI-algebras. The first one is an old theorem by I. Kaplansky \cite{IK}, dealing with primitive algebras.

\begin{theorem}\label{tIK}
Let $R$ be a primitive algebra, satisfying a polynomial identity of degree $d$. Then $R$ is simple, the center of $R$ is a field and $\dim_{Z(R)}R\le [d/2]^2$.
\end{theorem}

A very useful theorem due to E. C. Posner \cite{ECP} deals with prime algebras.

\begin{theorem}\label{tPosner}\emph{(E. C. Posner \cite{ECP})} Let $R$ be a prime algebra  over a field $\FF$, satisfying a nontrivial polynomial identity of degree $d$. Let $C$ be the center of $R$ and $Q$ the field of quotients of $C$. Then the algebra $R^Q=Q\ot_C R$ of central quotients of $R$ is central simple of dimension $n^2$, $n=[d/2]^2$, over its center $K$. Moreover, the identities of $R$ are the same as the identities of $R^Q$ and the same as the identities of $M_n(K)$. If $R$ has no zero divisors, then $R$ is a subalgebra of a division algebra of dimension $n^2$ over its center.
\end{theorem}

\section{Cocommutative Hopf algebras}\label{ssCCHA}

A well-known theorem on pointed cocommutative Hopf algebras is the following.

\begin{theorem}
Let $H$ be a pointed cocommutative Hopf algebra over a field $\FF$. Let $G = G(H)$ be the
group of group-like elements of $H$ and $H_0$ the connected component of 1. Then $G$ acts on
$H_0$ by conjugation and $H$ is isomorphic to the smash product $H_0\# \FF G$ via $h\# g \mapsto hg$ for $h\in H_0$, $g\in G$.
\end{theorem}

An algebra $H$ over a field $\FF$ is PI if and only if its extension to the algebraic closure of $\FF$ is PI. Since the condition of pointedness is automatically satisfied if $\FF$ is algebraically closed,
this theorem essentially reduces cocommutative Hopf algebras to the smash products of
connected cocommutative Hopf algebras and group algebras, where the group acts by
Hopf algebra automorphisms.  

Following Cartier (see \cite{AGI}), one uses the term \textit{hyperalgebra} for any connected cocommutative bialgebra. The
existence of the antipode is automatic for such bialgebras (see \cite[2.2.8]{BA}), so they are in
fact Hopf algebras.
It is well known (see \cite[5.6]{SM}) that in characteristic 0 any hyperalgebra $H_0$ is
isomorphic to the universal enveloping algebra $U(L)$ of the Lie algebra $L = P(H_0)$ of its
primitive elements.
So the question of polynomial identities of a cocommutative Hopf algebra of
characteristic 0 reduces to the study of the smash product of a universal enveloping algebra
$U(L)$ and a group algebra $\FF G$, where $G$ acts on $L$ by automorphisms.

We next recall some
known results about identities of cocommutative Hopf algebras. 

If $\FF$ is of characteristic zero, then a group algebra $\FF G$ is PI if and only if $G$ has an abelian subgroup of finite index \cite{IP}. An enveloping algebra $U(L)$ of a Lie algebra $L$ is PI if and only if $L$ is abelian \cite{L,B}.  A smash product $H=U(L)\#\FF G$ is PI if and only if $L$ is abelian and $G$ has an abelian normal subgroup $A$ of finite index which acts trivially on $L$ \cite{HLS,MKC}. 

These results were used in \cite{MKH} to prove the following.
\begin{theorem}\label{tCCHA}
Let $H$ be a cocommutative Hopf algebra of characteristic 0. Then the following conditions are equivalent: 
\begin{enumerate}
\item $H$ is PI as an algebra, 
\item There exists a normal commutative subHopfalgebra $A \subset H$ such that $H/HA^+$ is finite-dimensional, 
\item There exists a normal commutative subHopfalgebra $B\subset H$ such that $H$ is a finitely generated left $B$-module, 
\item The Lie algebra $L = P(H)$ of primitive elements is abelian and there exists a normal subHopfalgebra $C\subset \mathrm{corad}(H)$ of finite index such that $C$ is commutative and the adjoint action of $C$ on $L$ is trivial. 
\end{enumerate}

\end{theorem} 

In the case of fields of positive characteristic $p>0$, a group algebra $\FF G$ is PI \cite{P} if and only if $G$ contains a subgroup of finite index whose  commutator subgroup is a finite $p$-group. 

The question of when an arbitrary hyperalgebra over a field of positive characteristic is PI, remains largely open. Some cases where the answer is known are listed below.

An enveloping algebra $U(L)$ is PI  \cite{BA} if and only if there is an abelian subalgebra $M$ of finite codimension in $L$ and all inner derivations are algebraic of finite bounded degree. The last condition means that there is a natural number $n$ such that each inner derivation $\ad\,x$, $x\in L$, is annihilated by a polynomial of degree $n$.

A restricted enveloping algebra $u(L)$ of a restricted Lie algebra $L$ is PI  \cite{PE,PV} if and only if $L$ has restricted ideals $N\subset M\subset L$ such that $\dim L/M,\: \dim N\le\infty$, $[M,M]\subset N$, $[N,N]=0$ and there is  natural $n$ such that $x^{[p^n]}=0$, for each $x\in N$. 

If $H$ is a reduced hyperalgebra over a perfect field $\FF$ (that is, $H^\ast$ has no nilpotent elements), then $H$ is PI \cite{MKP} if and only if $H$ is commutative.

In the case of a smash product $U(L)\#\FF G$, over a field $\FF$ of characteristic $p>0$, this algebra is PI (\cite{BP}, see also \cite{HLS}) if and only if 
 \begin{enumerate}
 \item there exists an abelian $G$-invariant ideal $H\subset L$ of finite codimension and
all inner derivatives are algebraic of bounded degree,
\item there exists a normal subgroup $A\subset G$ of finite index such that the commutator
subgroup $[A,A]$ is a  finite abelian $p$-group,
\item $A$ acts trivially on $L$.
\end{enumerate}

In the case of $u(L)\#\FF G$, $\chr{\FF}=p>0$, this algebra is PI (\cite{BP}) if and only if 
\begin{enumerate}
 \item there are $G$-invariant restricted ideals $N\subset M\subset L$ such that
 \begin{enumerate}
 \item $\dim L/M,\:\dim N<\infty$,
 \item $M/N$ and $N$ are abelian,
 \item the $p$-map on $N$ is nilpotent,
\end{enumerate}  
\item there is a subgroup $A\subset G$ such that
\begin{enumerate}
\item $|G:A|<\infty$,
\item $[A,A]$ is a finite $p$-group,
\end{enumerate}
\item $A$ acts trivially on $M/N$.
\end{enumerate}

In the general case of connected Hopf algebras, Mikhail Kochetov \cite{MKP} conjectured that the following is true.

\begin{conjecture} Let $H$ be a hyperalgebra over a perfect field $\FF$. If $\chr{\FF} > 0$, assume also
that $H$ is reduced. Let $G$ be a group acting on $H$ by bialgebra automorphisms. Then the smash product $H\# \FF G$ is PI if and only if 
\begin{enumerate}
\item $H$ is commutative,
\item  there exist normal subgroups $G_0 \subset G_1  \subset G$ such that $G/G_1$ is finite, $G_1/G_0$ is
abelian, and $G_0$ is a finite $p$-group if $\chr{\FF} = p > 0$ and trivial if $\chr{\FF} = 0$,
\item  $G_1$ acts trivially on $H$.
\end{enumerate}
\end{conjecture}

In \cite{MKP}, the author proves that the conjecture is true if $H=H^{\mathrm{gr}}$ where the latter algebra is the associated graded algebra for $H$ with respect to the coradical filtration on $H$.

\section{Braided Hopf algebras}\label{ssNCCHA} Let $T$ be an abelian group and let $\beta:T\times T\to\FF^\times$ be a skew-symmetric bicharacter, sometimes called ``color''. A $T$-graded algebra $L=\bigoplus_{t\in T} L_t$ with a commutator product $(x,y)\to [x,y]$ is called a \textit{$\beta$-Lie superalgebra}, or a \textit{color Lie superalgebra}, if the following hold for any $x\in L_t$, $y\in L_u$, $z\in L_v$:
\begin{eqnarray}
&[x,y]=\beta(t,u)[y,x]\mbox{ (\textit{color anticommutativity})},\label{ecac}\\
&[[x,y],z]=[x,[y,z]]-\beta(t,u)[y,[x,z]]\mbox{ (\textit{color Jacobi identity}).\label{cJI} }
\end{eqnarray}
A $T$-graded associative algebra $A=\bigoplus_{t\in T} A_t$ becomes a color Lie superalgebra if one sets 
\begin{equation}\label{eCC}
[x,y]_\beta=xy - \beta(t,u)yx, \mbox{ for any homogeneous }x\in A_t,\:y\in A_u.
\end{equation}

Given a color (or $\beta$-) Lie superalgebra $L$, the (universal) enveloping algebra $U(L)$ for $L$ is defined as follows. First of all, $U(L)$ is a $T$-graded associative algebra, generated by $L$, which is a $\beta$-subsuperalgebra of $U(L)$ (that is, closed under $\beta$-commutator (\ref{eCC})). Second, given a $T$-graded  associative algebra $A$ and any homomorphism of $\beta$-superalgebras $\vp:L\to A$, there exists a (unique) homomorphism of associative algebras $f:U(L)\to A$ whose restriction to $L$ is $\vp$.

Given an alternating bicharacter $\beta:T\times T\to \FF^\times$, for any $t\in T$, we must have either $\beta(t,t)=1$ and then $t$ is called \textit{even}, or $\beta(t,t)=-1$, and then $t$ is called \textit{odd}. The set of all even elements in $T$ is a subgroup $T_+$, the set of odd elements is a coset of $T_+$, which we denote by $T_-$. Respectively, if $V=\bigoplus_{t\in T}V_t$ is a $T$-graded space then we write $V_+=\bigoplus_{t\in T_+}V_t$ and $V_-=\bigoplus_{t\in T_-}V_t$. If $L$ is a $\beta$-superalgebra with $L=L_+$, then we call $L$ a $\beta$-Lie algebra.

A theorem from \cite{BMPZ} says that in the case of characteristic zero, if $L$ is a $\beta$-Lie algebra, then $U(L)$ is PI if and only if $L$ is abelian. If $L$ is a general $\beta$-superalgebra $L=L_+\oplus L_-$, then $U(L)$ is PI if and only if there exists a homogeneous $L_+$-submodule $M\subset L_-$ such that $\dim L_-/M,\: \dim[L_+,M]<\infty$, and $[L_+,L_+]=[M,M]=0$. 

If $L$ is a $\beta$-Lie \textit{superalgebra}, then $L$ is a Yetter-Drinfeld module over $\FF T$, that is, $L\in ^{\FF T}_{\FF T}\!\!\mathcal{YD}$, where the action of $T$ on $L$ is given by $t\ast x_u=\beta(t,u)x_u$. The same is true for $U(L)$. The smash product $U(L)\#\FF T$ becomes a Hopf algebra, called the \textit{bosonization} of $U(L)$.

A generalization of such algebras is the smash product of the form  $U(L)\#\FF G$, where $G$ is a group (not necessarily abelian) acting on $L$ (hence on $U(L)$) by $T$-graded automorphisms: $g\ast L_t=L_t$, for any $g\in G$ and any $t\in T$.

There are no papers where the polynomial identities of such bosonizations for $\beta$-Lie superalgebras have been examined. But there is a result of M. Kochetov \cite{MK}, where $\ch(\FF)=0$, $L$ is an ordinary Lie superalgebra, $L=L_0+L_1$, and $G$ is an arbitrary group (not just $G=T$). Kochetov shows that $H=U(L)\# \FF G$ is a PI-algebra if and only if  there exist $L_0$- and $G$-invariant subspaces $N\subset M\subset L_1$ and an abelian subgroup of finite index $A\subset G$ such that $\dim L_1/M, \dim N<\infty$, $[L_0,L_0]=[M,M]=0$, $[L_0,M]\subset N$, and $A$ acts trivially  on $L_0$ and $M/N$.

The following result, although quite expected, is nevetherless new.

\begin{theorem}\label{tLSABW}
Let $\FF$ be a field of characteristic~0.
Let $\beta: T\times T\to\FF^\times$ be an alternating bicharacter on a finite abelian group $T$,  $L$ a $\beta$-Lie superalgebra, with finite support, $G$ a group acting on $L$ by $T$-graded automorphisms, $U(L)$ the enveloping algebra for $L$, $\FF G$ the group algebra for $G$. Then $U(L)\#\FF G$ is a PI-algebra if and only if the following hold.
\begin{enumerate}
\item $L_+$ is abelian;
\item There is an abelian subgroup $A$ of finite index in $G$ which acts trivially on $L_+$;
\item There is a $G$-invariant $L_+$-submodule $M$ in $L_-$ such that $[M,M]=0$ and $\dim[L_+,M]<\infty$;
\item $A$ acts trivially on $M/[L_+,M]$.
\end{enumerate}
\end{theorem}

We give a proof of this result below in Section \ref{sPT}.

If $\FF$ is a field of positive characteristic and $L$ is a color Lie superalgebra or color Lie $p$-superalgebra over $\FF$, the conditions under which $U(L)\#\FF G$ or $u(L)\#\FF G$ is a PI-algebra are not known.

\section{Techniques}\label{sDS}

\subsection*{Group-theoretical prelude}

If a \textit{group $G$ has a subgroup $A$ of finite index $m$ whose commutator subgroup $[A,A]$ is either $1$ or of order $p^k$ in case $\chr{\FF}=p>0$}, then the group algebra $R=\FF G$ satisfies a nontrivial polynomial identity. The reason is that in this case $R$ can be viewed as a free left $S$-module of rank $m$, where $S=\FF A$. The right regular action of $R$ on itself imbeds $R$ in the matrix algebra $M_m(S)\cong M_m(\FF)\ot S$. By Regev's Theorem \cite{AR}, stated here as Theorem~\ref{tAR}, the tensor product of two PI-algebras is PI. If $A$ is abelian, then $S$ is commutative ($[x,y]=0$ holds in $S$), hence $R$ is PI. If $A$ is $p$-abelian, that is, $[A,A]=p^k$ and $\chr{\FF}=p>0$, then one easily shows that $[x,y]^{p^k}=0$ is satisfied in $S$. Again, $R$ is PI. 

The above condition on $G$ is also necessary for $\FF G$ to satisfy a nontrivial polynomial identity but proving this needs much more sophistication. 

If $G$ is a group then one defines the Delta set $\Delta(G)$ to be the characteristic subgroup consisting of the elements having only finitely many conjugates. In other words, $\Delta(G)$ is the union of the set of all finite congugacy classes. A group $G$ with $G=\Delta(G)$ is called an \textit{FC-group}. Given a natural number $k$, the Delta set $\Delta_k(G)$ is the set of all elements in $G$ whose conjugacy classes have at most $k$ elements. An important theorem of B.H. Neumann and J. Wiegold \cite{W} says that if $k$ is a natural number and $G$ is a group whose commutator subgroup $[G,G]$ is finite of order $k$ then $G=\Delta_k(G)$; conversely, if $G=\Delta_k(G)$ then $|[G,G]|<(k^4)^{k^4}$.

If a group $G$ has an abelian (normal) subgroup $A$ of finite index then $A\subset\Delta(G)$ so that $\Delta(G)$ has finite index in $G$. In \cite{S} using Posner's Theorem \ref{tPosner}, M. Smith has shown that if $\FF G$ is a \textit{prime} group ring satisfying a nontrivial polynomial identity then $G$ has an abelian subgroup of finite index. In \cite{P} D. S. Passman has shown, without any conditions on the primeness of $\FF G$, that if $\FF G$ satisfies a polynomial identity of degree $n$ and $k=(n!)^2$  then the \textit{set} $\Delta_k(G)$ has index in $G$ not exceeding $(k+1)!$. So if $\FF G$ is PI, then the existence of a subgroup of finite index with finite commutator subgroup follows. 

Passman also produced an example in characteristic $p>0$ showing that the existence of an abelian subgroup of finite index is not necessary in the case of the fields of prime characteristic~\cite[Theorem 4.5]{PL}.

\subsection*{Lie algebras}
A special feature of Lie algebras, compared with groups, is that, in the case $\chr{\FF}=0$, every nonabelian Lie algebra $L$ has an infinite-dimensional irreducible representation $\rho:L\to\End\,V$. This extends to an irreducible representation of $U(L)$. Using Kaplansky's Theorem \ref{tIK}, we can conclude that $U(L)$ is a PI-algebra in the case $\chr{\FF}=0$ if and only if $L$ is abelian. 

Another feature is that the enveloping algebra $U(L)$ of a Lie algebra $L$ over any field has no zero divisors. In particular, $U(L)$ is always a prime algebra. So Posner's Theorem~\ref{tPosner} applies.
This theorem works not only in the case of enveloping algebras of Lie algebras but also in other situations, such as enveloping algebras of color Lie algebras or some smash products, which can be viewed as skew group rings of groups with coefficients in an algebra. An important paper related to this approach is \cite{HLS}.

In the case of enveloping algebras of Lie algebras over a field of characteristic $p>0$, see \cite{BA}, where Posner's Theorem~\ref{tPosner} was used in combination with the technique of Delta sets, as defined below.

Let $L$ be a Lie algebra over the field $\FF$ and let $U(L)$ denote its universal enveloping algebra. In case $\chr{\FF}= p > 0$, assume that $L$ is restricted and $u(L)$ is 
its restricted enveloping algebra. Then $U(L)$, $u(L)$, and the group
ring $\FF[G]$, for a group $G$, are all Hopf algebras and hence are similar in many ways.
In particular, since questions on group algebras have been solved using
$\Delta$-methods, it was therefore reasonable to try to find similar techniques in the
Lie context. To this end one considers
\begin{equation}\label{eDelta}
\Delta=\Delta(L)=\{x\in L\,|\,\dim_{\,\FF}\,[x,L]<\infty\},
\end{equation}
the (restricted) Lie ideal of $L$ introduced in \cite{BA}. In the same way, as in the case of groups, we have Delta sets 
\[
\Delta_n(L)=\{x\in L\,|\,\dim_\FF[x,L]\le n\}.
\]
Clearly, $\Delta(L)$ is an ideal in $L$. The sets $\Delta_n(L)$ are not ideals. Still, one says that $\Delta_n(L)$ has codimension $m$ in $L$ if there is an $m$-dimensional subspace $V$ in $L$ such that $L=\Delta_n(L)\oplus V$ and $m$ is the minimal number with this property. 

It was shown in \cite{BA} that if some $\Delta_n(L)$ is of codimension $m$ in $L$, then $L$ has a subalgebra $M$ such that $L/M$ and $[M,M]$ are finite-dimensional. Here in place of the B. H. Neumann - Wiegold theorem in the case of groups one uses a general result about bilinear maps due to P. M. Neumann \cite{PMN}: 
\begin{theorem}\label{tPMN}
If $f:U\times V\to W$ is a bilinear map such that for each $u\in U$ one has $\dim_\FF\,f(u,V)\le m$ and for each $v\in V$ one has $\dim_\FF\,f(U,v)\le n$ then $\dim_\FF f(U,V)\le mn$.
\end{theorem}

\subsection*{Lie superalgebras and smash products}

We conclude that if a Delta set for a group or Lie algebra is of finite index or codimension then this strongly affects the structure of  the group algebra of the group or (restricted) enveloping algebra of the Lie algebra in question. This remains true when one considers color Lie superalgebras. However, in the case of a color Lie superalgebra $L=\bigoplus_{t\in T}L_t$, defined by an alternating bicharacter $\beta:T\times T\to\FF^\times$, one has to consider ``graded'' Delta sets as follows.  
\begin{definition}\label{dDeltas} For any $t,u\in T$, $m\in\NN$ one sets
\begin{enumerate}
\item $\Delta^m_{t,u}(L)=\{x\in L_t\,|\,\dim\, [x,L_u]\le m\}$;
\item $\Delta_t^m(L)=\bigcap_{u\in T}\Delta^m_{t,u}(L)$;
\item $\Delta_t(L)=\bigcup_{m\in\NN}\Delta^m_t(L)$;
\item $\Delta(L)=\bigoplus_{t\in T}\Delta_t(L)$.
\end{enumerate}
\end{definition}

Let $q\in\FF$ and consider a $q$-commutator $[a,b]_q=ab-qba$. An important result about PI-algebras is the following (see \cite[Theorem 4.2.3]{BMPZ}).
\begin{theorem}
Let $L=\bigoplus_{t\in T}L_t$ be a color  (restricted) Lie superalgebra over an arbitrary field $\FF$ and fix $t,u\in T$. Suppose that in $U(L)$ or $u(L)$ some nontrivial $q$-polynomial $f(x_1,\ld,x_n,y_1,\ld,y_n)$ where $q=\beta(u,t)$, satisfies
\[
f(x_1,\ld,x_n,y_1,\ld,y_n)=\sum_{\sigma\in\mathrm{Sym}(n)}\lambda_\sigma[x_1,y_{\sigma(1)}]_q\cdots[x_n,y_{\sigma(n)}]_q=0,\: \lambda_\sigma\in\FF,
\]
for any $x_1,\ld,x_n\in L_u$, $y_1,\ld,y_n\in L_t$. Then any $n$ elements in $L_t$ are linearly dependent modulo $\Delta^m_{t,u}(L)$.
\end{theorem}

Repeated application of this theorem together with Theorem \ref{tPMN} above on bilinear maps allows one to obtain an $L_+$-submodule $M\subset L_-$, with the desired properties in the case where $U(L)$ satisfies a polynomial identity and $\chr{\FF}=0$ (see Theorem \ref{tLSABW} as well as its proof in the last section). In this case also $U(L_+)$ is a prime algebra and so Posner's Theorem~\ref{tPosner} applies, which easily implies that $L_+$ must be abelian.

Also, using Theorem \ref{tH_inner_id} in the case of $U(L)\#\FF G$, with $x_1,\ld,x_n\in G$ together with additional Delta sets allows one to prove in Theorem \ref{tLSABW} the condition on the existence of the abelian subgroup $A$.

\subsection*{Delta sets in Hopf algebras}

The definition of Delta sets in general Hopf algebras was given in \cite{BePa2}, as follows.

Let $H$ be a Hopf algebra over a field $\FF$ of arbitrary characteristic. 
Let us define the ``Delta set'' $H_{\mathrm{fin}}$ by setting
\begin{equation}\label{eHfin}
H_{\mathrm{fin}}=\{ x\in H\,|\,\dim_{\FF}(\ad\,H)(x)<\infty\}.
\end{equation}
Since the adjoint action is a measuring, $H_{\mathrm{fin}}$ is an $\FF$-subalgebra of $H$, but not necessarily a Hopf subalgebra.

If $T$ is a subset of $H$ which generates $H$ as an $\FF$-algebra then $x\in H_{\mathrm{fin}}$ if and only if $x$ is contained in an $\ad\,T$-stable finite-dimensional subspace of $H$.

The Delta sets $\hf$ are fairly closely related to the Delta sets defined earlier, for groups and Lie algebras.
\begin{proposition}[J. Bergen - D.S. Passman \cite{BePa2}\label{pBP}] The following are true:
\begin{enumerate}
\item If $H=\FF\, G$ is a group algebra of a group $G$ then $\hf=\FF\Delta(G)$;
\item If $L$ is a restricted Lie algebra over a field $\FF$ with $\chr{\FF}=p>0$ and $H=u(L)$ its restricted enveloping algebra then $\hf=u(\Delta(L))$;
\item Let $L$ be a Lie algebra over a field $\FF$ with $\chr{\FF}=0$, $H=U(L)$ its enveloping algebra and $\Delta_L$ the join of all finite-dimensional ideals of $L$. Then $\hf=U(\Delta_L)$.
\end{enumerate}

\end{proposition}

The authors of \cite{BePa2} mention that $U(L)_{\mathrm{fin}}$ can be appreciably smaller than $U(\Delta(L))$. An explicit example is given in the case $\chr{\FF}=p>0$.

In \cite{MK} M. Kochetov studies Delta sets in the case of smash products $A\# H$, where $A$ is an $H$-module algebra, $H$ a Hopf algebra, under the condition that $A$ is generated by an $H$-submodule $V$, which generates $A$ as a unital algebra. This restriction is natural in all cases of pointed Hopf algebras considered here. He sets
\begin{eqnarray*}
\delta_H^m(V)&=&\{v\in V\,|\, \dim(H\cdot v)\le m\},\; \delta_H(V)=
\bigcup_m\delta_H^m(V),\\
\delta_V^m(H)&=&\{h\in H\,|\, \dim(h\cdot V)\le m\},\; \delta_V(H)=
\bigcup_m\delta_V^m(H).
\end{eqnarray*}

Some properties of these Delta sets are similar to those of Delta sets in the case of group algebras, enveloping algebras and their smash products. For instance,
\begin{enumerate}
\item For any $\alpha,\beta\in\FF$, $x\in\delta_H^i(V)$, $y\in\delta^j(V)$, one has $\alpha x+\beta y\in\delta_H^{i+j}(V)$;
\item All $\delta_H^m(V)$ are $H$-invariant sets;
\item If $\vp:V\to V$ is a homomorphism of $H$-modules then $\vp(\delta_H^m(V))\subset\delta_H^m(V)$;
\item If $\alpha,\beta\in\FF$, $h\in\delta_V^i(H)$, $k\in\delta_V^j(H)$ then $\alpha h+\beta k\in\delta_V^{i+j}(H)$;
\item All $\delta_V^m(H)$ are invariant under the left and the right multiplication by the elements of $H$.
\end{enumerate}
As before, $\delta_V(H)$ is a two-sided ideal in $H$ but not necessarily a Hopf ideal; $\delta_H(V)$ is an $H$-submodule in $V$.

Since Delta sets do not need to be subspaces, we give the following  definition. Suppose that $W$ is a subset in a vector space $V$. We say that $W$ has finite codimension in $V$ if there exist $v_1,\ld,v_m \in V$ such that $V = W +\Span\{v_1,\ld,v_m\}$ . If $m$
is the minimum possible integer with such property then we set $\dim V/W = m$.
We also introduce the notation $m\cdot W =\{ w_1+\cdots + w_m\;|\;w_i \in W\}$, $m \in \NN$. The following is \cite[Lemma 6.3]{BP}. 
\begin{lemma}\label{lDIM} Let $L$ be a vector space. Suppose that a subset $W$ is stable
under multiplication by scalars and such that $\dim_{\FF} L/W \le m$. Then $\Span_{\FF}{W}= 4^m\cdot W$.
\end{lemma}

Using the generating set $V$ for $A$ one can consider the associated graded algebra $\gr\,A$ defined by the degree filtration $\{A_m\}_{m=0}^\infty$ of $A$ with respect to $V$. Here $A_0=\FF 1$ and, for $m>0$, we have $A_m=A_{m-1}+V^m$. Then  
\[
\gr A=\bigoplus\gr_mA,\mbox{ where }\gr_mA=A_m/A_{m-1}.
\]
All information about the action of $H$ on $A$ is contained in the action of $H$ on $\gr\,A$. Let us denote by $S(V)$ the symmetric algebra of $V$, $\Lambda(V)$ the Grassman algebra of $V$ and, if $\chr{\FF}=p>0$, $V^{[p]}$ will stand for the subspace in $S(V)$ spanned by all $v^p$, where $v\in V$. We quote the following result of M. Kochetov \cite[Proposition~2.4]{MK}.

\begin{theorem}\label{tMKfil}
Let $A$ be an $H$-module algebra generated by an $H$-submodule $V$ as a unital algebra. Assume that the associated graded algebra $\gr A$ is isomorphic to one of $S(V)$, $\Lambda(V)$ or $S(V)/\mathrm{ideal}(V^{[p]})$. If the polynomial identity (\ref{eH_inner_id}) holds for any $h_1,\ld,h_n\in H$ and $y_1,\ld,y_n\in V$ then
\begin{enumerate}
\item $\dim V/\delta_H^{n^2}(V)<n$;
\item $\dim H/\delta_V^{n^2}(H)<n$.
\end{enumerate}
\end{theorem}

If $H=\FF G$ is a group algebra and $A=U(L)$ where $G$ is a group and $L$ is a color Lie superalgebra then this theorem can be used to obtain Theorem \ref{tLSABW}.
Details are in the last section, making the observation that the theorem generalizes to color Lie superalgebras $A$.

\section{Identities and Delta sets in Radford's algebra $F_{(q)}$}\label{sFQ}

A very particular case  of algebras described in the previous section are the Hopf algebras $F_{(q)}$ introduced by D.~Radford~\cite{Rad}. The algebra $H=F_{(q)}$, $q\in\FF$, is generated by one group-like element $a$ and one skew-primitive element $x$ such that 
\begin{equation}\label{eQQ}
ax=qxa.
\end{equation}  
The algebras $F_{(q)}$ are standard 2-generator subalgebras in the pointed Hopf algebra $U(\cD,\lambda)$ introduced in the next section. In this section we study polynomial identities and Delta sets in $F_{(q)}$.

\subsection*{Polynomial identities in $F_{(q)}$}

\begin{theorem}\label{pPIFQ}
The algebra $H=F_{(q)}$ satisfies a nontrivial polynomial identity if and only if $q$ is a root of 1. If $q$ is a primitive $n$th root of unity, then $H=F_{(q)}$ satisfies the standard identity 
\[
S_{2n}(x_1,\ld,x_{2n})=\sum_{\sigma\in\mathrm{Sym}(2n)}\sgn(\sigma)x_{\sigma(1)}\cdots x_{\sigma(2n)}=0,
\]
and $2n$ is the minimum degree of identities in $F_{(q)}$.
\end{theorem}

\begin{proof}
In \cite[Section 4.1.6]{BMPZ}, the authors construct an irreducible module for a color polynomial algebra in two variables. Similarly, one should consider here a vector space $V$ with basis $\{ v_i\vert i\in\ZZ\}$. We will make $V$ an $F_{(q)}$-module by setting
\begin{equation}\label{eV}
a\circ v_i=v_{i-1},\; x\circ v_i=q^iv_{i+1}\mbox{ for any }i\in\ZZ.
\end{equation}
\begin{proposition}\label{pIMFQ}
If $q$ is not a root of 1, then $V$ is an irreducible $F_{(q)}-$module whose centralizer is $\FF\,\id_V$.
\end{proposition}

\begin{proof}
Let $W$ be a nonzero submodule in $V$. If some $v_i$ is in  $W$ then clearly $W=V$.  Otherwise, let us take a nonzero element $w$ in $W$, which has the form
\[
w=\sum_{i=0}^m \mu_iv_{t_i},\mbox{ where }0\ne \mu_i\in\FF,\;m\ge 0,\;t_i\ne t_j, \mbox{ if }i\ne j.
\]
We apply induction on $m$ starting with $m=0$ to show that $W=V$. Indeed, let us consider
\[ 
(ax)\circ w=\sum_{i=0}^m q^{t_i}\mu_iv_{t_i}\mbox{ for all }0\le s\le m.
\]
Subtracting this element from $q^{t_m}w$, we obtain
\[
q^{t_m}w-(ax)\circ w=\sum_{i=0}^{m-1} \mu_i(q^{t_m}-q^{t_i})v_{t_i}\in W.
\]
Since all coefficients in this new element are nonzero, induction applies and hence $V$ is irreducible.
If $\vp$ is in the centralizer of this module then 
\[
\vp(a\circ v_i)=a\circ(\vp(v_i)\mbox{ and }\vp(x\circ v_i)=x\circ(\vp(v_i)\mbox{ for all }i\in\ZZ.
\]
If $\vp(v_i)=\sum_{j\in\ZZ}t_i^jv_j$ then we have, for all $i,j$, that $t_i^j=t_{i-1}^{j-1}$ and $q^jt_i^j=q^it_{i+1}^{j+1}$. Since $q^{j-i}=1$ if and only if $j=i$ we have $\vp(v_i)=t^i_iv_i$. But also all $t^i_i$ are equal, and so $\vp=\lambda\,\id_V$, for some $\lambda\in\FF$. Thus the centralizer of $F_{(q)}$-module $V$ is $\FF\,\id_V$, as claimed.

\end{proof}

Note that, in a quite similar way to the proof of Proposition \ref{pIMFQ}, if $q$ is a primitive $n$th root of unity, one can construct an irreducible module $V_n$ for $F_{(q)}$, whose dimension is $n$. One simply has to take $V_n=\langle  v_i\vert i\in\ZZ_n\rangle$ and use the same formulas (\ref{eV}), but view addition modulo $n$. The same proof shows that the following is true.

\begin{proposition}\label{pIM_nFQ}
If $q$ is a primitive $n$th root of 1, then $V_n$ is an irreducible $F_{(q)}$-module of dimension $n$. The centralizer of $V_n$ is $\FF\,\id_{V_n}$.
\hfill$\Box$
\end{proposition}

Continuing with the proof of the theorem, let us assume that $q$ is not root of 1. Since $H=F_{(q)}$ has an irreducible module of infinite dimension, by the Density Theorem, for any $n$, a homomorphic image of $F_{(q)}$ contains a subalgebra isomorphic to the matrix algebra $M_n(\FF)$ of order $n$. By Amitsur-Levitzky's Theorem, the minimum degree of identities satisfied by $M_n(\FF)$ equals $2n$. As a result, $F_{(q)}$ is not a PI-algebra.

If $q$ is a primitive $n$th root of 1, then it follows from (\ref{eQQ}) that both $a^n$ and $x^n$ are central elements in $H=F_{(q)}$. Moreover, it follows from \cite{Rad} (PBW-basis for $F_{(q)}$) that  $H$ is a free  left module of rank $n$ over the commutative subalgebra $K=\FF[a,x^n]$. (One could also take $K=\FF[a^n,x]$, etc.) Let $r_u$ denote right multiplication in $H$ by $u\in F_{(q)}$, an endomorphism of the free left $K$-module $H$. Since $1\cdot r_u=u$ and $r_{u_1u_2}=r_{u_1}r_{u_2}$, the mapping $u\mapsto r_u$ is an injective homomorphism of $H$ to $\End_K(H)$. Thus $H$ embeds in the matrix algebra $M_{n}(K)$ and so satisfies a standard polynomial identity of degree $2n$.

On the other hand, since $H$ has an irreducible module $V_n$, by the Density Theorem its homomorphic image is isomorphic to $M_n(\FF)$, hence by  Amitsur-Levitzky's Theorem $F_{(q)}$ cannot satisfy an identity of degree  less than $2n$.
 \end{proof}

\begin{corollary} \label{cq_root}
If $q$ is a primitive $n$th root of 1, then $F_{(q)}$ satisfies an identity of degree $2n$ but no identity of degree less than $2n$.
\end{corollary}

\subsection*{Delta sets in $F_{(q)}$}

Let $\FF$ be an arbitrary field, $q\in \FF^{\times}$.
Let $H= F_{(q)}$ as above (see~\cite[\S1.3]{Rad}). 
Note that  
\[    
  \Delta(x) = x\otimes a + 1\otimes x \ \ \mbox{ and } \ \ S(x)= - x a^{-1} . 
\]

Recall the definition of Delta sets from Section~\ref{sDS},
under the adjoint action of $H$ on itself:  
\[
    \delta^i_H(H) = \{ h\in H\mid \dim_{\FF}(\ad H)(h)\leq i \}
\]
and 
\[
   H_{{\rm fin}} = \bigcup_{i=0}^{\infty} \delta^i_H(H) = \{ h\in H \mid
      \dim_{\FF} (\ad H)(h) < \infty \}. 
\]

We will prove that $H_{{\rm fin}} = H$ if, and only if, $q$ is a root of unity,
and that if $q$ is a primitive $n$th root of unity, then $\delta^n_H(H) = H$. 
First we need a calculational lemma.

\begin{lemma}\label{lad}
For all integers $i\geq 0$ and all integers $j$, 
\begin{itemize}
\item[(i)] $\displaystyle{  (\ad x^i)(a^j) = \prod_{l=0}^{i-1} (1-q^{j-l}) x^i a^{j-i}}$ 
  and    $(\ad a^i)(a^j) = a^j$, 
\item[(ii)] $(\ad x^i)(x^j) = 0$  and $(\ad a^i) (x^j) = q^{ij}x^j$.
\end{itemize}
\end{lemma}

\begin{proof}
We will prove the first formula in part (i) by induction on $i$.
First note that 
\[
   (\ad x)(a^j) = xa^j a^{-1} + a^j S(x)
   = xa^{j-1} - a^j x a^{-1} = (1-q^j) x a^{j-1} . 
\]
Next assume that 
\[
    (\ad x^{i-1})(a^j)  = \prod_{l=0}^{i-2} (1-q^{j-l}) x^ia^{j-i}
\]
for some $i\geq 2$.
Then 
\begin{eqnarray*}
   (\ad x^i) (a^j) & = & (\ad x) (\ad x^{i-1})(a^j) \\ & = &
    \prod_{l=0}^{i-2} (1-q^{j-l}) xx^{i-1} a^{j-i+1} a^{-1}  
    - x^{i-1} a^{j-i+1} x a^{-1}) \\
   & = & \prod_{l=0}^{i-1} (1-q^{j-l}) (x^i a^{j-i}) .
\end{eqnarray*}
Therefore the first formula in part (i) holds. 
Clearly the second formula also holds. 

Note that $(\ad x)(x^j) = xx^ja^{-1}  - x^j x a^{-1} = 0$.
Thus the first formula in part (ii) holds.
The second is straightforward.
\end{proof}

\begin{theorem}
Let $H= F_{(q)}$. Then 
$H_{{\rm fin}} = H$ if, and only if, $q$ is a root of unity.
Moreover, if $q$ is a primitive $n$th root of unity, 
then $\delta^n_H(H) = H$. 
\end{theorem}

\begin{proof}
Suppose $q$ is a primitive $n$th root of unity.
Then by Lemma~\ref{lad}, for each $i,j$, 
\[
   (\ad x^{n+i}) (a^j) = 0 .
\]
Therefore $(\ad x^{n+i}a^{i'}) (a^j) = 0$ for all $i,i'$, and 
similarly for any linear combination of elements of the form
$x^{n+i}a^{i'}$.
It now follows from Lemma~\ref{lad} that
$\dim_k((\ad H)(a^j)) = n$, 
and so $a^j\in \delta_n(H)$.
Further, $\dim_k ((\ad H)(x^j))=1$, so $x^j\in \delta_1(H)$. 
Consequently, 
$\dim_k((\ad H)(x^ia^j)) =n$ 
for all $i\geq 0$ and  $j\geq 1$, since the
adjoint action is multiplicative, so
$x^ia^j \in \delta_n(H)$.  
Since the adjoint action is $k$-linear,
it now follows that $\delta^n_H(H) = H_{{\rm fin}} = H$.

Now suppose that $q$ is not a root of unity.
By Lemma~\ref{lad}, for all $i$,
\[
    (\ad x^i)(a) = \prod_{l=0}^{i-1} (1-q^{1-l}) x^i a^{1-i} .
\]
It follows that $a\not\in H_{{\rm fin}}$.
Therefore $H_{{\rm fin}}\neq H$.
\end{proof}

\section{The pointed Hopf algebras $U(\cD, \lambda)$ of Andruskiewitsch - Schneider}\label{pointed}

In this section we consider some infinite dimensional pointed Hopf algebras introduced by Andruskiewitsch and Schneider in their work on classification of finite dimensional pointed Hopf algebras~\cite{AS2}. Our algebras are somewhat more general, with fewer restrictions on the groups of grouplike elements. 

Let $\FF$ be a field of characteristic~0.
Let $\Gamma$ be a group and $\hat{\Gamma}$ its group of characters, that is group homomorphisms to $\FF^{\times}$, 
with identity element denoted $\varepsilon$. 
Let $\theta$ be a positive integer and $(a_{ij})_{1\leq i, j\leq \theta}$ a 
Cartan matrix of finite type. 
For each $i$, $1\leq i\leq \theta$, let $g_i$ be in the center of $\Gamma$ and $\chi_i\in \hat{\Gamma}$ such that $\chi_i(g_i)\neq 1$ and 
\[
    \chi_j(g_i) \chi_i(g_j) = \chi_i(g_i)^{a_{ij}}
\]
for all $i,j$. Let $q_{ij} = \chi_j(g_i)$. 
In case $\Gamma$ is finite abelian, we call 
\begin{equation}\label{edatum}
  \cD = (\Gamma, \ (g_i)_{1\leq i\leq \theta} , 
    (\chi_i)_{1\leq i\leq \theta}, \ (a_{ij})_{1\leq i,j\leq \theta} )
\end{equation}
a {\em datum of finite Cartan type}.

Let $V$ be the vector space with basis $x_1,\ldots,x_{\theta}$.
Define an action of $\Gamma$ on $V$ by 
\[
     g(x_i) = \chi_i(g) x_i
\]
and a coaction by $ \delta(x_i) = g_i\ot x_i$
for all $i$.
Then $V$ is a Yetter-Drinfeld module over $\FF\Gamma$. 
The tensor algebra $T(V)$ is a braided Hopf algebra in the
Yetter-Drinfeld category $^{\FF\Gamma}_{\FF\Gamma}{\rm {YD}}$, where
the braided coproduct takes $x_i$ to $x_i\ot 1 + 1\ot x_i$ for all $i$.
Define braided commutators:
\[
     \ad_c (x_i) (y) = x_i y - g_i(y) x_i
\]
for all $y\in T(V)$. 
Choose scalars $\lambda_{ij}$, $1\leq i,j\leq \theta$, $i\not\sim j$
(i.e.\ $i,j$ not in the same connected component of the Dynkin diagram)
for which $\lambda_{ij}=0$ if $g_ig_j =1$ or $\chi_i\chi_j\neq \varepsilon$.
Let $H=U(\cD, \lambda)$ be the Hopf algebra 
defined as in~\cite{AS2} to be the quotient of the
smash product $T(V)\# k\Gamma$ by relations
\[
\begin{aligned}
& \mbox{{\em (Serre)}} \hspace{1cm}   (\ad_cx_i)^{1-a_{ij}}(x_j) = 0 \quad 
    (i\neq j,  \ i\sim j) , \\
& \mbox{{\em (linking)}} \hspace{1cm} (\ad_cx_i)(x_j) = \lambda_{ij} (1-g_ig_j)
   \quad   (i<j, \ i\not\sim j) . 
\end{aligned}
\]
The coalgebra structure on $U(\cD, \lambda)$ is given  by 
\[
    \Delta(g) = g\ot g , \ \ \  \Delta(x_i) = x_i\ot 1 + g_i\ot x_i ,
\]
$\varepsilon(g) =1$, $\varepsilon(x_i)=0$, $S(g) = g^{-1}$,
and $S(x_i) = - g_i^{-1}x_i$ for all $g\in\Gamma$, $1\leq i\leq\theta$.
Thus we see that $\ad_c$ as defined above agrees with the
adjoint action of the Hopf algebra $H=U(\cD,\lambda)$ on itself. 

Using the root system $\Phi$ associated with  the Cartan matrix $(a_{ij})$, one can build finitely many elements $x_\beta$, $\beta\in \Phi^+$, as braided commutators in the algebra $T(V)$. If $\Gamma$ is a finite abelian group, then by \cite[Theorem 3.3]{AS2}, $H$ has a PBW basis determined by the root vectors corresponding to positive roots. We use this theorem to establish the following.
\begin{theorem}\label{tASany}  For arbitrary $\Gamma$,
if all $\chi_i(g_j)$ are roots of unity, then 
\begin{enumerate}
\item $H$ has a PBW basis consisting of elements
\[
x_{\beta_1}^{a_1}x_{\beta_2}^{a_2}\cdots x_{\beta_p}^{a_p}g,\mbox{ where } a_i\ge 0, \ g\in \Gamma.
\]
\item For any $\alpha,\beta\in\Phi^+$, there is a positive integer
$N_{\beta}$ and a scalar $q_{\alpha\beta}$ (given by products of values of $\chi_i(g_j)$) such that the elements $x_\beta^{N_\beta}$ generate
a Hopf ideal and 
\[
x_\alpha x_\beta^{N_{\beta}}=q_{\alpha\beta}^N x_\beta^{N_{\beta}} x_\alpha.
\]
Moreover, a power (at least $N_\beta$) of $x_{\beta}$ commutes with any other element. 
\end{enumerate}

\begin{proof} Denote 
\[
\Gamma_{{\mathbf{\chi}}} = \bigcap_{i=1}^{\theta} \mathrm{Ker}(\chi_i).
\]
Let us consider the datum $\overline{\cD}$ given by
\begin{equation}\label{edatum_bar}
  \overline{\cD} = (\overline{\Gamma}, \ (\bg_i)_{1\leq i\leq \theta} , 
    (\overline{\chi}_i)_{1\leq i\leq \theta}, \ (a_{ij})_{1\leq i,j\leq \theta} )
\end{equation}
where $\Gamma$ is replaced by $\overline{\Gamma}=\Gamma/\Gamma_\chi$, $\bg=g\Gamma_\chi$, $\overline{\chi}_i(\bg)=\chi_i(g)$. By the hypotheses and by its definition, $\overline{\Gamma}$ is necessarily finite abelian. 
One checks that this datum gives rise to an Andruskiewitsch-Schneider Hopf algebra $\bH=U(\overline{\cD},\lambda)$, where the vector space is again $V$ but with the natural Yetter-Drinfeld module structure over $\FF\overline{\Gamma}$. Since $\overline{\Gamma}$ is finite abelian, we can invoke \cite[Theorem~3.3]{AS2}, according to which $\bH$ is of finite codimension over its braided center. More precisely, $\bH$ has a PBW basis
\[
x_{\beta_1}^{a_1}x_{\beta_2}^{a_2}\cdots x_{\beta_p}^{a_p}\bg,\mbox{ where } a_i\ge 0, \ \bg\in \overline{\Gamma}.
\] 
Also \cite[Theorem~3.3]{AS2} ensures existence of the numbers $N_\beta$ such that $x_\beta^{N_\beta}$ is in the braided center of $\bH$. In addition, since all $\overline{\chi}_i(\bg_j)$ are necessarily of finite order, for each $\beta$ there is some integer (at least $N_\beta$) such that $x_\beta$ raised to this power commutes with any other element.

Considering the defining relations for $H$ and $\bH$, one easily checks that there is a ``natural'' homomorphism $\vp:H\to\bH$ such that $\vp(u\# g)=u\#\bg$. Using this homomorphism, the PBW-basis of $\bH$ and the numbers mentioned above for $\bH$, one finds that both statements of our theorem are true.
\end{proof}

\end{theorem}

We give several small examples to illustrate the ubiquity of the Hopf algebras
$U(\cD,\lambda)$. Their polynomial identities and delta sets will be taken up in the next section.

\begin{example}
($U_q({\mathfrak{sl}}_3)^{\geq 0}$)
Let $\ell$ be an odd positive integer and $\Gamma = \ZZ/\ell\ZZ\times \ZZ/\ell \ZZ$
(respectively $\Gamma = \ZZ\times \ZZ$), with generators $g_1,g_2$.
Let $q$ be a primitive $\ell$th root of unity in $\CC$ 
(respectively, let $q$ be any element of $\CC$ other than $0,1,-1$).
Consider the type $A_2$ Cartan matrix,
\[
    \left(\begin{array}{rr} 2 & -1 \\ -1 & 2 \end{array}\right)
\]
Define characters $\chi_1,\chi_2$ by
\[
  \begin{array}{rl} \chi_1(g_1) = q^2 , & \hspace{.5cm} 
    \chi_1(g_2) = q^{-1} , \\
     \chi_2(g_1) = q^{-1} , & \hspace{.5cm} \chi_2(g_2) = q^2 .
\end{array}
\]
Then $U(\cD,0)\cong R\# \CC\Gamma$ where $R$ is the algebra
generated by $x_1,x_2$ with relations
\[
   x_{12} x_1 = q^{-1} x_1 x_{12} , \  \ \ x_2 x_{12} = q^{-1} x_{12}x_2 , 
\]
where $x_{12}:= [x_1,x_2]_c = x_1x_2 - q^{-1}x_2x_1$.
The algebra $R$ has PBW basis $\{ x_1^a x_{12}^b x_2^c \mid a,b,c\geq 0\}$. 
In case $\Gamma$ is infinite, this is the
quantum group $U_q({\mathfrak{sl}}_3)^{\geq 0}$.
In case $\Gamma$ is finite, a quotient by the ideal generated by
powers of $x_1,x_{12},x_2$ is the small
quantum group $u_q({\mathfrak{sl}}_3)^{\geq 0}$. 
\end{example}

\begin{example}
($U_q({\mathfrak{sl}}_2)$)
Let $\ell$ be an odd positive integer and $\Gamma =  \ZZ/\ell \ZZ$
(respectively $\Gamma = \ZZ$), with generator $g$.
Let $q$ be a primitive $\ell$th root of unity in $\CC$ 
(respectively, let $q$ be any element of $\CC$ other than $0,1,-1$).
Consider the type $A_1\times A_1$ Cartan matrix,
\[
    \left(\begin{array}{rr} 2 & 0 \\ 0 & 2 \end{array}\right)
\]
Let $g_1=g_2=g$ and 
define characters $\chi_1,\chi_2$ by
$\chi_1(g) = q^{-1} , \ 
     \chi_2(g) = q $.
Let $\lambda_{12} = \frac{q}{q-q^{-1}}$. 
Then $U(\cD,\lambda)$ is the quotient of the smash product
$T(V)\# \CC\Gamma$, where $V$ is a vector space with
basis $x_1,x_2$, by the ideal generated by
\[
     x_1x_2 - q x_2 x_1 - \lambda (1-g^2) .
\]
In case $\Gamma$ is infinite, this is isomorphic to
$U_q({\mathfrak{sl}}_2)$.
(An isomorphism is given by $x_1\mapsto E$, $x_2\mapsto K^{-1}F$,
$g\mapsto K^{-1}$.)
In case $\Gamma$ is finite, a quotient by the ideal generated by
powers of $x_1,x_2$ is the small quantum group $u_q({\mathfrak{sl}}_2)$. 
\end{example}

\begin{example}
(Generalization of $F_{(q)}$ to more generators \cite{AS})
Let $\Gamma = \ZZ\times\cdots\times\ZZ$ ($n$ copies) with
generators $g_1,\ldots, g_n$. 
Consider the type $A_1\times\cdots\times A_1$ Cartan matrix.
Choose nonzero scalars $q_{ij}$ ($i\leq j$).
Define characters $\chi_1, \ldots, \chi_n$ by 
$\chi_i(g_j) = q_{ji}$ for $i\leq j$ and
$\chi_i(g_j) = q_{ij}^{-1}$ for $i>j$. 
Let $\lambda_{ij}=0$ for all $i,j$. 
Then $U(\cD,0)\cong R\# k\Gamma$ where $R$ is the algebra
generated by $x_1,\ldots,x_n$ with relations
\[
   x_i x_j = q_{ij} x_j x_i .
\]
The quotient by the ideal generated by all $x_i^{N_i}$
(where $N_i = o (q_{ii})$) is the quantum linear space of \cite{AS}.

\end{example}

\section{Identical relations and Delta sets in $U(\cD,\lambda)$}\label{ssIR} 

In this section, we determine necessary and sufficient conditions for $U(\cD,\lambda)$, defined in Section~\ref{pointed}, to satisfy a polynomial identity, or to be equal to its delta set.

\subsection*{Identities in  $U(\cD,\lambda)$}

\begin{theorem}\label{A-S-PI}
An algebra $H=U(\cD,\lambda)$ satisfies a nontrivial polynomial identity if and only if 
\begin{enumerate}
\item $\Gamma$ has an abelian normal subgroup $\widetilde{\Gamma}$ of finite index, which contains the center of $\Gamma$, hence the elements $g_1,\ld,g_\theta$;
\item the orders of all characters $\chi_1,\ld,\chi_\theta$ are finite.
\end{enumerate}
\end{theorem}

\begin{proof} Assume that $H$ satisfies a nontrivial polynomial identity. 
Since $\CC \Gamma$ is a subalgebra in $H$, one can use  D. S. Passman's Theorem \cite{P} to conclude that such $\widetilde{\Gamma}$ of finite index in $\Gamma$ exists. So  Condition (1) must hold. Now if the order of some $\chi_i$ is infinite then for any $m$ there is $g\in\Gamma$ such that the order of $q=\chi_i(g)$ is greater than $m$. In this case, $H$ contains a subalgebra isomorphic to $F_{(q)}$ such that by Corollary \ref{cq_root} this subalgebra does not satisfy a polynomial identity  of degree less than $2m$. As a result, $H$ does not satisfy a nontrivial polynomial identity.  This contradiction shows that Condition (2) must also hold.

Now assume that both conditions are satisfied. By Theorem \ref{tASany}, $H$ is a free right module of finite type over a subalgebra $K$ generated by $\{x_\beta^{N_\beta}\,|\,\beta\in\Phi_+\}$, and $\Gamma_\chi$. At the same time $\CC\Gamma_\chi$ is a free right module of finite type over a commutative subalgebra $\CC(\Gamma_\chi\cap\widetilde{\Gamma})$. If $S$ is the subalgebra generated by $\{x_\beta^{N_\beta}\,|\,\beta\in\Phi_+\}$ and $\Gamma_\chi\cap\widetilde{\Gamma}$, then $S$ is commutative and $H$ is a free right module of finite type over the commutative subalgebra $S$. It follows that $H$ is a PI-algebra.

\end{proof}

\subsection*{Delta sets in  $U(\cD,\lambda)$}

Recall from Section~\ref{sDS} that a group $G$ is an \textit{$FC$-group} if all conjugacy classes in $G$ are finite. In other words, $\Delta(G)=G$.

\begin{theorem}\label{tHFIN} 
Let  $H=U(\cD,\lambda)$.
Then $H_{{\rm{fin}}} = H$ if and only \begin{enumerate}
\item $\Gamma$ is an FC-group;
\item the orders of all characters $\chi_1,\ld,\chi_\theta$ are finite.
\end{enumerate}
\end{theorem} 

\begin{proof}
Assume our Conditions (1) and (2) hold. 
Then all $q_{ij}$ are roots of unity. 
By Theorem~\ref{tASany}, 
there are positive integers $M_{i}$ such that the
elements $x_{\beta_1}^{M_{1}},\ldots, x_{\beta_p}^{M_p}$
are central in $U(\cD,\lambda)$, and  
$\ad x_{\beta_1}^{M_1}, \ldots \ad x_{\beta_p}^{M_p}$
are zero operators.
Since for each $g\in G$, $\ad g$ applied to any PBW basis element 
$x_{\beta_1}^{a_1}\cdots x_{\beta_{p}}^{a_p}$ 
is multiplication by a root of unity, under conditions (1) and (2), 
\[
   \{   \ad(x_{\beta_1}^{b_1}\cdots x_{\beta_p}^{b_p} g ) (h) \mid
     b_1,\ldots, b_p\geq 0; g\in G \}
\]
is a finite set for each $h\in H$. 
That is, $H_{\rm{fin}} = H$.

Now assume  $H_{\rm{fin}}= H$. Then also $(\CC G)_{\rm{fin}} = \CC G$. By Proposition \ref{pBP}, we have that $(\CC G)_{\rm{fin}} = \CC \Delta(G)$. So we must have $\CC \Delta(G)=\CC G$. It follows that $G=\Delta(G)$, so Condition (1) is satisfied. 

If condition (2) does not hold, then as we saw in the proof of
Theorem~\ref{A-S-PI}, for any $m$, $H$ has a subalgebra
isomorphic to $F_{(q)}$ where $q$ has order
greater than $m$. By Lemma~\ref{lad}(i), there exists an element
of $H$ that is not in $\delta^m_H(H)$.
Therefore $H_{{\rm{fin}}}\neq H$.
\end{proof}

\section{Proof of Theorem \ref{tLSABW}}\label{sPT}

\subsection*{Sufficiency}
Let us first prove that the  conditions of Theorem \ref{tLSABW} guarantee that an algebra $R=U(L)\#\FF G$ is a PI-algebra. For this, we note that $L_+\oplus M$ is a $G$-invariant ideal in $L$. Let us consider in $R$ the subalgebra
\[
R_1=U(L_+\op M)\#\FF A.
\]
It follows from the PBW Theorem for color Lie superalgebras (see \cite[Chapter 3, Theorem 2.2]{BMPZ} that $R$ is a left $R_1$-module generated by the elements of the form
\begin{equation}\label{e8}
x_{i_1}\cdots x_{i_s}\# g_j,\;i_1<\cdots<i_s, s=1,2,\ld,
\end{equation}
where $\{x_i\}$ is a basis of $L_-$ modulo $M$ and $\{g_j\}$ is a right transversal for $A$ in $G$. Since the number of elements of the form (\ref{e8}) is finite, it is sufficient to prove that $R_1$ is a PI-algebra. 

Let us consider a two-sided ideal $I$ of $R_1$ generated by $N=[L_+,M]$. Since $N$ is an $A$-invariant ideal in the superalgebra $L_+\op M$, we can see that 
\[
I=R_1N=NR_1.
\]
By the PBW Theorem, $N^{\dim N+1}=\{ 0\}$, so that $I^{\dim N+1}=\{ 0\}$, showing that $I$ is a nilpotent ideal. Moreover,
\[
R_1/I\cong U(L_+\op M/N)\#\FF A\cong U(L_+\op M/N)\ot\FF A
\]
since the actions of $A$ and of $L_+$ on $M/N$ are trivial. Now $L_+\op(M/N)$ satisfies the conditions of \cite[Chapter 4, Theorem 1.2]{BMPZ} and so $U(L_+\op M/N)$ is a PI-algebra. Since $\FF A$ is commutative, $R_1$ is indeed a PI-algebra.
\subsection*{Necessity} 
If $U(L)\# \FF G$ is PI, then also $U(L)$ and $\FF G$ are PI and so we have an abelian subgroup $A_1$ of finite index in $G$ and an $L_+$-submodule $M$ in $L_-$, satisfying the following conditions:
\begin{equation}\label{e9}
[L_+,L_+]=0,\;[M,M]=0,\; L_-/M\mbox{ and } [L_+,M]\mbox{ are finite-dimensional.}
\end{equation}
Now we need to consider the action of $G$ on $L$. Without loss of generality, we may assume that $G$ is abelian and $[L_-,L_-]=0$. 

First we look at $U(L_+)\#\FF G$. Let us denote by $p\ast x$ the result of the action of $p\in\FF G$ on $x\in U(L_+)$. The existence of a subgroup of finite index in $G$, which trivially acts on $L_+$, can be recovered from the proof of \cite[Theorem 2.3]{BP}. In that proof just a couple of lines need changing. 
First of all, instead of a polynomial ring $U(L)$ we need to deal with a color polynomial ring $U(L_+)$, where $L_+$ is an abelian color Lie algebra. This is a $T$-graded vector space and $xy=\beta(t,u)yx$ for homogeneous elements $x$ of degree $t$ and $y$ of degree $u$. Now $G$ acts by $T$-graded automorphisms and so for any $p,q\in\FF G$, $p\ast x$ is still homogeneous of degree $t$ and $q\ast y$ is still homogeneous of degree $u$. Now $U(L_+)$ is still an integral ring but $\beta$-commutative, not commutative. The only place in the the proof in question where the commutativity is used in that paper is on p.~374, after equation (4). In the lines of that paper that follow, one needs to show that if, for some $p_i,q_i\in\FF G,\:x,y\in U(L_+)$ one has
\[
(p_1\ast x)(q_1\ast y)+\cdots+(p_m\ast x)(q_m\ast y)=0,
\] 
then one also has
\[
(q_1\ast x)(p_1\ast y)+\cdots+(q_m\ast x)(p_m\ast y)=0.
\]
Choosing $x,y$ homogeneous of degree $t,u$, respectively and using $\beta$-commutativity proves that this can be done in the case of $U(L_+)$. The triviality of the action of a subgroup of finite index on all homogeneous elements of $L_+$ implies the triviality of its action on the whole of $L_+$.

Thus we may assume that we deal with the smash products $U(L)\#\FF G$ where $G$ is abelian acting trivially on the whole of $L_+$. Now we proceed to the study of the action of $G$ on $L_-$. Since $U(L_-)\# \FF G\subset U(L_-)\# \FF G$, we may assume that $U(L_-)\# \FF G$ is a PI-algebra and so there is a nontrivial identical relation for the action of $G$ on $U(L_-)$ such that $L_-$ is a $G$-invariant subspace. Since $[L_-,L_-]=0$, we know from the PBW Theorem for color Lie superalgebras (see \cite[Chapter 3, Theorem 2.2]{BMPZ} that $U(L)$ is a color Grassmann algebra $\Lambda^\beta(L_-)$, where $\beta:T\times T\to \FF^\times$ is an alternating bicharacter such that $\beta(t,t)=-1$ for any $t\ne 1$ in $T$. We will use the proof of \cite[Theorem 2.14]{MK} (where $L$ is an ordinary superalgebra, that is, $L_+=L_0$ and $L_-=L_1$) to prove the following more general result. 

\begin{proposition}\label{pT2.4} Let $L=L_+\op L_-$ be a color Lie superalgebra over a field $\FF$ of characteristic zero such that $R=U(L_-)\#\FF G$ is a PI-algebra. If there is an identity holding in $R$ of the form
\begin{equation}\label{ee1}
\sum_{\pi\in \mathrm{Sym}(n)}\gamma_\pi(h_1\ast X_{\pi(1)})\cdots(h_n\ast X_{\pi(n)})=0
\end{equation}
where $h_1,\ld,h_n\in\FF G$, $X_1,\ld,X_n\in L_-$, then there is a subgroup $G_1$ of finite index in $G$, and a $G$-invariant subspace $M$ of finite codimension in $L_-$ and $N$ of finite dimension, such that the action of $G_1$ on $M/N$ is trivial.
\end{proposition}

In the proof of this result Kochetov uses Theorem \ref{tMKfil}. In this theorem one deals with algebras $A\# G$ generating a subspace $V$ such that the associated graded algebra $\gr A$ is one of $S(V)$, $\Lambda(V)$ or $S(V)/\mathrm{ideal}(V^{[p]})$. However, the proof only uses the PBW-bases for $A$ in question. Since the PBW Theorem holds for color Lie superalgebras, without any changes the proof applies to $A=U(L)$, in particular to color Grassmann algebras. 

So we have the following.

\begin{proposition}\label{pMKfil} If (\ref{ee1}) holds in $U(L_-)$ for any $h_1,\ld,h_n\in\FF G$ and $X_1,\ld,X_n\in L_-$, then
\begin{enumerate}
\item $\dim L_-/\delta_H^{n^2}(L_-)<n$;
\item $\dim H/\delta_{L_-}^{n^2}(H)<n$.
\end{enumerate}
\end{proposition}

Note that in the case $H=\FF G$, the first inequality enables us to obtain a $G$-invariant subspace $M$ of codimension at most $n$ in $L_-$, which is contained in $\delta_{\FF G}^{4^n n^2}(L_-)$. Likewise, the second inequality enables us to obtain an ideal $I$ of codimension at most $n$ which is contained in  $\delta_{L_-}^{4^n n^2}(\FF G)$. If we apply P. M. Neumann's Theorem~\ref{tPMN} to the bilinear map $I\times M\to I\ast M$, given by $(p,w)\mapsto p\ast w$, for $p\in I$, $w\in M$, we obtain
\begin{equation}\label{ee2}
\dim (I\ast M)\le 16^nn^4.
\end{equation} 

Some additional work is needed to obtain the subgroup $G_1=\delta_{L_-}(G)$ of finite index in $G$ and a finite-dimensional $G$-invariant subspace $N$ such that the action of $G_1$ on $M/N$ is trivial. In the case of ordinary superalgebras, this is done in \cite[Proposition 2.13]{MK}. The ``color'' version of this proposition, adjusted to our needs, is as follows.

\begin{proposition}\label{p2.13}
Let $G$ be an abelian group acting on $A=\Lambda^\beta(V)$, where $\beta$ is an alternating bicharacter on a finite abelian group $T$, $V$ a $T$-graded vector space. If the identity of the action of degree $n$ (\ref{ee1}) holds  for any $h_1,\ld,h_n\in\FF G$ and any $X_1,\ld,X_n\in A$ then $[G:\delta_V(G)]<d$ where $d$ depends only on $n$.
\end{proposition}

The proof of Proposition 2.13 in \cite{MK} is based on Lemmas 2.9, 2.10, 2.11 and Proposition 2.12, describing the restrictions on the $T$-graded action of a cyclic group $(g)$ on $A=\Lambda^\beta(V)$ if the identity of the action is satisfied. The proofs of the lemmas remain unchanged if all the elements of $V$ on which $g$ acts are assumed $T$-homogeneous. If two elements $x,y$ present in the calculations are of the same degree $t\in T$ then $yx=\beta(t,t)xy=-xy$ and so the calculations work in the same way as in the case of the ordinary $\Lambda(V)$. The final conclusion of a preliminary Proposition~2.12 in \cite{MK} word by word translates to the color situation and provides us with the number $c$, depending only on $n$, such that, for any $g\in G$, $\dim((g^c-1)\ast V)<c$. 

Finally, the derivation of Proposition 2.13, hence of our Proposition \ref{p2.13}, translates to the color case without problems. One only needs to keep in mind that the spaces, algebras and the actions appearing in the proofs in \cite{MK} are $T$-graded.

Let us denote by $G_1$ the subgroup $\delta_{L_-}(G)$, which is of finite index thanks to Proposition \ref{p2.13}. Remember the ideal $I$ appearing in (\ref{ee2}). If we denote by $(\FF G)^+$ the augmentation ideal of $\FF G$, then 
\[
\dim(\FF G_1I\cap (\FF G_1)^+\le \dim(\FF G)/I\le n.
\]
We can then choose $g_1,\ld,g_n\in G_1$ such that
\[
(\FF G_1)^+=\Span(g_1-1,\ld,g_n-1)+I\cap(\FF G_1)^+.
\]
It then follows that
\[
(\FF G_1) \ast M\subset (g_1-1)\ast M+\cdots+(g_n-1)\ast V+I\ast M.
\]
By construction, the right hand side of the latter expression is finite-dimensional.

The completion of the proof of the classification theorem for ordinary Lie superalgebras \cite[Section 3]{MK} does not depend on the color structure and can be used as the completetion of the proof of our Theorem \ref{tLSABW}. \hfill$\Box$

\end{document}